\documentclass[11pt]{amsart}
\usepackage[active]{srcltx}
\usepackage{mathrsfs}
\usepackage[T1]{fontenc}

\usepackage{latexsym,enumerate}

\usepackage{amsmath,amssymb,amsthm,amsfonts,latexsym}
\usepackage[bookmarks]{hyperref}
\hypersetup{backref, colorlinks=true, colorlinks   = true,
urlcolor     = blue, linkcolor = blue, citecolor   = magenta
}
\hypersetup{backref, colorlinks=true}

\def\adh#1{\overline{#1}}

\let\=\partial

\newtheorem {pro}{Proposition}[section]
\newtheorem {thm}[pro]{Theorem}
\newtheorem {cor}[pro]{Corollary}
\newtheorem{lem}[pro]{Lemma}

\theoremstyle{definition}
 \newtheorem {rem}[pro]{Remark}
\newtheorem {dfn}[pro]{Definition}

\DeclareMathOperator*{\argmin}{arg\,min}

\newcommand{\tra}{\mathbf{tr}}

\newcommand{\R}{\mathbb{R}}

\newcommand{\N}{\mathbb{N}}
\newcommand{\cc}{\mathscr{C}}\newcommand{\cbb}{\mathbf{C}}
\newcommand{\st}{\mathscr{S}}

\newcommand{\et}{\quad \mbox{and} \quad }

\newcommand{\hn}{\mathcal{H}}

\newcommand{\mba}{ {\overline{M}}}

\newcommand{\wbf}{\mathbf{W}}\newcommand{\nub}{\mathbf{\nu}}\newcommand{\gab}{\mathbf{\gamma}}\newcommand{\gub}{\gab_\nub}

\newcommand{\lbf}{\mathbf{L}}
\newcommand{\ubf}{\mathbf{u}}

\newcommand{\ep}{\varepsilon}

\newcommand{\pa}{\partial}

\newcommand{\orn}{{0_{\R^n}}}

\newcommand{\pau}{\pa}
\newcommand{\pad}{\pa}

\newcommand{\xo}{{x_0}}
\newcommand{\mc}{{\check{M}}}

\newcommand{\supp}{\mbox{\rm supp}}

 \thanks{Research partially supported
by the NCN grant  2021/43/B/ST1/02359.}
\title[]{On the Laplace equation on bounded subanalytic manifolds}

\makeatletter

\@namedef{subjclassname@2020}{%
\textup{2020} Mathematics Subject Classification}
\makeatletter

\@addtoreset{equation}{section}

\makeatother

\author[ G. Valette]{ Guillaume Valette}

\address[G. Valette]{Instytut Matematyki Uniwersytetu
Jagiello\'nskiego, ul. S. \L ojasiewicza 6, Krak\'ow, Poland}\email{guillaume.valette@im.uj.edu.pl}



\begin{document}

\begin{abstract}
We prove a trace formula for integration by parts on subanalytic bounded submanifolds of $\R^n$, possibly non closed. We also establish density results for $\mathbf{W}^{1,p}_\nabla (M)$, $M$ bounded subanalytic manifold, which is the space of the $L^p$ tangent vector fields $v$ on $M$ for which $\nabla v$ is $L^p$, where $\nabla$ is the divergence operator. We derive from these results some theorems of existence and uniqueness of solutions of the Laplace equation  with Dirichlet and Neumann boundary type conditions. We then study the $p$-Laplace equation, for $p\in [1,\infty)$ large.
\end{abstract}

\maketitle
\begin{section}{Introduction}The recent works on the Lipschitz geometry of subanalytic sets \cite{m, pa, vlt, gvpoincare, lipsomin, projreg, livre} enabled to develop the theory of Sobolev spaces of subanalytic manifolds \cite{poincfried, poincwirt,trace, lprime}, which makes it possible to investigate partial differential equations on such manifolds. 
In this article, we illustrate this by giving some theorems of existence and uniqueness  of solutions of the Laplace equation on singular varieties.

 Although subanalytic manifolds may fail to be metrically conical at frontier points, several basic theorems, such as Sobolev or Morrey's embedding, Poincar\'e inequality, trace theorems, and density of smooth compactly supported functions, were proved to have natural generalizations to the Sobolev space $W^{1,p}$ of such a manifold, with sometimes some restrictions on $p$. An application to the study of the Laplace equation, with Dirichlet boundary conditions, was already provided in \cite{lprime}.  We give in this article a study of the Neumann condition.

We first focus on the Sobolev space $W^{1,p}(M)$  of a bounded subanalytic submanifold $M\subset \R^n$ as well as on the space $\wbf_\nabla^{1,p}(M)$ of the $L^p$ tangent vector fields $v$ on $M$ for which $\nabla v$ is $L^p$, where $\nabla$ is the divergence operator, showing some density theorems for ``nicely supported functions'' in this space (section $3$). The idea is  to derive these results from the density theorems of \cite{trace, lprime} (see section $2$) by duality.  Using classical arguments \cite{t}, we then deduce existence and uniqueness of a cotrace operator $\gub$ and prove a formula for integration by parts involving the trace operator of the regular part of the boundary (identity (\ref{eq_trace_theorem})). As in \cite{lprime}, we nevertheless have to restrict ourselves to the values of $p$ that  are not greater than the codimension of the singularities of the boundary (see Remark \ref{rem_traceformula}). The singularities of the boundary are always of codimension at least $2$, which makes it possible to cover the case $p=2$, which is crucial to investigate partial differential equations.

We therefore give applications to the classical problem 
   $$\begin{cases}
 \Delta u=g \qquad \qquad\mbox{ on } M,\\	
 \tra_{\pa M} u= \tra_{\pa M} f \;\; \mbox{ on } \Gamma_D, \\
\; \gab_\nub \pa u=\theta \qquad \;\quad \mbox{ on } \Gamma_N,
\end{cases}$$
where $\Delta$ is the Laplacian, $\tra_{\pa M}$ is the trace operator on the regular part of the frontier of $M$ (see section $2$ for $\pa M$) and $\gub$ is the cotrace operator that we construct in section $4$ (Theorem \ref{thm_trace_formula}).  Here, $\Gamma_D$ and $\Gamma_N$ are open subanalytic subsets of $\pa M$ such that $\dim \pa M\setminus (\Gamma_D\cup \Gamma_N)\le\dim M-2$. We establish existence and uniqueness of the solution in $W^{1,2}(M)$ (Theorem \ref{thm_neumanndir}).

 The subanalytic category is valuable for applications for it encloses the trigonometric and exponential functions, as well as all the sets defined by polynomial equalities and inequalities, which naturally emerge from engineering problems.

 Many mathematical problems that emanate from engineering, especially from machine learning \cite{bn,bert, bishop, buhl,elmo,elmo2,zhu,luxb}, involve finding a function that takes prescribed values on a given set while minimizing the $L^p$ norm of its gradient, i.e. of determining, given a sufficiently regular function $f$ on a subset $A$ of $\mba$,
 \begin{equation}\label{eq_argmin_intro}\argmin_{u\in W^{1,p}(M), u= f \mbox{ on $A$ }} ||\pa u||_{L^p(M)}.\end{equation}
 We thus end this article by showing existence and uniqueness of the solution to this optimization problem in our framework for every $p$ sufficiently large (Theorem \ref{thm_opt}).  It is well-known \cite{degen} that, on a manifold $M$ that has sufficiently regular boundary, the solutions of the nonlinear equation  $\Delta_p u=0$, where $\Delta_p$ is the $p$-Laplacian, are the functions $u$ for which $||\pa u||_{L^p(M)}$ is minimal among all the functions that have the same trace as $u$. We show that this property holds on any bounded subanalytic manifold $M$ for sufficiently large values of $p$, even if $M$ admits singularities in its closure.

 It seems that one could extend the results to the case where $p$ is not greater than the codimension of the singularities. The case ``$p$ large'' seems however to be the most relevant for applications, for the convergence of a sequence in the $W^{1,p}$ norm guarantees uniform convergence and H\"older continuity of the limit, via Morrey's embedding \cite{lprime}. Furthermore, it enables to prescribe the values of $u$ on a set $A\subset \mba$ of any codimension.
It was actually established \cite{alam,elal,nadl} that the convergence of the methods that provide solutions to this kind of problems demands to choose $p\ge \dim M$ or $\dim M+1$. Our work, together with the results of \cite{lprime} (or rather their proofs), emphasizes that higher values of $p$ will be necessary in the case where the underlying manifold $M$ admits singularities within its closure.  The convergence seems to be actually more correlated to the Lipschitz geometry of the pair $(M,A)$ than to the dimension of $M$. For instance, if  $M$ is  the set $x^2+y^2=z^k$,  $z>0$ (in $\R^3$, with $k>2$), then $p> k$ should be a more relevant choice.

For simplicity we assume that $M$ is a connected.
 We refer the reader to  \cite{bm, ds, l, livre} for all the basic facts about subanalytic geometry. The results of this article are actually valid for all bounded manifolds that are definable in some polynomially bounded o-minimal structure expanding the field $\R$ \cite{costeomin,vdd}.

\section{Some notations, conventions, and basic facts}
\begin{subsection}{Main notations and conventions.}
 Throughout this article,   the letter 
$M$  stands for a  bounded connected oriented subanalytic $\cc^\infty$ submanifold of $\R^n$, for some $n\in \N$, and $m$ for its dimension.
   We set $\dim \emptyset=-\infty$.

By ``manifold'', we will always mean submanifold of $\R^n$. Except otherwise specified, a submanifold of $\R^n$ will always be endowed with its canonical measure, provided by volume forms, and we will not indicate the measure when integrating on a submanifold of $\R^n$ with respect to this measure. Here is a glossary of the main notations.

$x\cdot y$  and $|x|$ :  euclidean inner product of $x$ and $y$ and euclidean norm of $x$
    
 $\adh A$ : closure of a set $A\subset \R^n$
 
 $\delta A:=\adh A\setminus A$, if $A\subset \R^n$

 $||u||_{L^p(M)}$  : $L^p$ norm  of  $|u|$ (with respect to the measure of $M$)
 
  $L^p(M)$ (resp. $\lbf^p(M)$) : space of $L^p$ measurable functions on $M$ (resp.  vector fields on $M$ tangent to $M$)

 $\nabla \beta$ : divergence of a tangent vector field $\beta$ on $M$, defined as $\nabla \beta=*d *\beta$, where $d$ is the exterior differentiation (as current) and $*$ the Hodge operator

    $\supp\, u$ : support of a current $u$ on $M$

      $p'$ : the H\"older conjugate of $p\in (1,\infty)$, i.e., $p'=\frac{p}{p-1}$, $1'=\infty$, and $\infty'=1$
    
   
 $\cc^\infty(E)$ (resp. $\cbb^{\infty}(E)$)   : for $E\subset \mba$,  space of those functions (resp. tangent vector fields) on $E\cap M$ that smoothly extend to an open neighborhood of $E$ in $\R^n$

  $\cc_0^\infty(M)$ (resp. $\cbb_0^{\infty}(M)$) :  space of elements of $\cc^\infty(M)$ (resp. $\cbb^{\infty}(M)$)  that are compactly supported

We set for $p\in [1,\infty]$:
    $$W^{1,p}(M):= \{u\in L^p(M),\; \partial u \in \lbf^p(M)\},$$
where  $\pa u$ stands for the  gradient of $u$,
     as well as
     $$\wbf_\nabla^{1,p} (M):=\{\beta\in \lbf^{p}(M):\nabla \beta \in L^p(M) \}. $$
We endow these two spaces with the norms:
$$||u||_{W^{1,p}(M)}:= ||u||_{L^p(M)}+||\pa u||_{L^p(M)} \et ||\beta||_{\wbf^{1,p}_\nabla(M)}:=||\beta||_{L^p(M)}+||\nabla \beta||_{L^p(M)}.$$
We also set
        $$\wbf^{1,\infty}(M):= \{u\in \lbf^\infty(M),\; \partial u \in \lbf^\infty(M)^n\}.$$

 By definition, $\cc^\infty(\mba)$ is a subset of $W^{1,p}(M)$.  In particular, for $u\in \cc^\infty(\mba)$,  $\supp\, u$ will be a subset of $M$ and
   we set for $V\subset \mba$ $$\cc^\infty_{V}(\mba):=\{u\in \cc^\infty(\overline{M}):\overline{ \supp\, u} \subset V  \}.$$
The spaces $\cbb^{\infty}_{V}(\mba)$, $W^{1,p}_V (M)$, $\wbf^{1,p}_V (M)$, $\lbf^p_V(M)$, and $\wbf_{\nabla,V}^{1,p} (M) $ are defined analogously.

 Given any couple of measurable functions or tangent vector fields $u$ and $v$ on $M$ for which suitable integrability conditions hold, we set:
$$<u,v>:=\int_M u\cdot v .$$

  $\hn^k$ : $k$-dimensional Hausdorff measure

%
%

\end{subsection}

 \subsection{Sobolev spaces with Dirichlet boundary conditions.} Let $A\subset\delta M$. 
If $u\in W^{1,p}(M)$ is zero in the vicinity of $A$ and if $\beta\in \wbf_\nabla^{1,p'}(M)$ is zero in the vicinity of $\delta M\setminus A$, it then easily follows from Stokes' formula that we have 
\begin{equation}\label{eq_vanishing}
 <\pa u, \beta>=-< u,\nabla\beta> .
\end{equation}
For this reason, this equality is often used to express the vanishing of functions and vector fields on the boundary in the weak sense. Let us set for $p\in[1,\infty]$:
 $$W^{1,p}(M,A) :=\{u\in W^{1,p}(M): \mbox{(\ref{eq_vanishing}) holds for all } \beta \in \wbf^{1,p'}_{\nabla,M\cup A}(M) \}.$$
 
Similarly, let
$$\wbf_\nabla^{1,p}(M,A) :=\{\beta\in \wbf_\nabla^{1,p}(M): \mbox{(\ref{eq_vanishing}) holds for all }  u\in W_{M\cup A}^{1,p'}(M) \}, $$
and note that  \begin{equation*}\label{eq_WMCUPA} W^{1,p}_{\mba\setminus  A}(M) \subset W^{1,p}(M,A) \et \wbf^{1,p}_{\nabla, \mba\setminus  A}(M) \subset \wbf_\nabla^{1,p}(M,A) .\end{equation*}
 
It is also worthy of notice that the mapping $\pa :W^{1,p}(M,A)\to \lbf^p(M)$ may fail to be injective even if $A\ne \emptyset$, i.e. that constant nonzero functions may vanish on a nonempty subset  of $\delta M$, and that this depends on $p$.  For instance, if $p\le 2$ and $\dim A\le m-2$ (and $A$ subanalytic) then Theorem \ref{thm_dense_lprime} below shows that  $\pa :W^{1,p}(M,A)\to \lbf^p(M)$ is not injective, while, if $p$ is large enough, it is injective for all nonempty subsets $A$ of $\pa M$ (we recall that $M$ is assumed  to be connected).

\subsection{Some basic facts.}
The following theorem, which extends to subanalytic domains one of the most famous features of Lipschitz domains, was established in \cite{lprime}.
\begin{thm}\label{thm_derlp_implique_lp}
 Let $u$ be a distribution on $M$ and let $p\in [1,\infty)$. If $|\pa u|$ is $L^p$ then so is $u$.
\end{thm}

This theorem makes it possible to derive the following fact.

\begin{cor}\label{cor_pa_nonrel}
 The mapping $\pa:W^{1,p}(M)\to \lbf^p(M)$, $u\mapsto \pa u$, has closed image for all $p\in [1,\infty)$.
\end{cor}
\begin{proof}
 It follows from de Rham's theory \cite{derham} that an element $v\in \lbf^p(M)$ is the gradient of a distribution $u:M\to \R$ in the distribution sense if and only if for every $\varphi \in \cbb_0^{\infty}(M)$ satisfying $\nabla \varphi =0$ we have
 \begin{equation}\label{eq_test_grad}
   <v,\varphi>=0.
 \end{equation}
 Let $v_i$ be a convergent sequence in the image of $\pa$, i.e., $v_i =\pa u_i$, with $u_i \in W^{1,p}(M)$ and $v_i$ tending to some $v\in \lbf^p(M)$ in the $\lbf^p$ norm. Equality (\ref{eq_test_grad}), which is valid for $v_i$, remains valid for $v$. Hence, $v=\pa u$, for some distribution $u$ on $M$. By Theorem \ref{thm_derlp_implique_lp}, as $|\pa u|$ is $L^p$,  $u$ must be $L^p$ as well.
\end{proof}

  Corollary \ref{cor_pa_nonrel} amounts to say that there is a constant $C$ such that for all $u \in W^{1,p}(M)$ we have
   \begin{equation}\label{eq_poinc_wirt}
    \inf_{a\in \R} ||u-a||_{L^p(M)} \le C||\pa u||_{L^p(M)}.
   \end{equation}
  In the case where $M$ is a  bounded subanalytic open subset of $\R^n$, this fact follows from the Poincar\'e-Wirtinger inequality proved in \cite{poincwirt}.

\end{section}

\section{Trace operators and density theorems in $W^{1,p}(M,A)$}

\subsection*{The set $\pa M$.} 
We will say that $M$ is {\bf Lipschitz regular} at $x\in \delta M$ if this point has a neighborhood $U$ in $\R^n$ such that each connected component of $U\cap M$ is the interior of a Lipschitz manifold with boundary $U\cap \delta M$. 
 We then let:
$$\pau M:=\{x\in \delta M: M\mbox{ is Lipschitz regular at $x$}\}.$$
This set is subanalytic and we have
\begin{equation}\label{eq_pa1}\dim \left(\delta M \setminus  \pau M\right)\le m-2.\end{equation}
 This fact follows from a famous result  which is sometimes referred as {\it Wing Lemma} by geometers (see \cite[Proposition 1, section 19]{l}, \cite[Proposition 9.6.13]{bcr}, or \cite[Lemma 5.6.7]{livre}).   
  We denote by $M_\sigma$ the singular part of the boundary, i.e.: $$M_\sigma:=\delta M\setminus \pa M.$$

  \begin{subsection}{Normal manifolds.}\label{sect_normal}We recall the notions of normal manifold and $\cc^\infty$ normalization,
  introduced in \cite{trace}.
  \begin{dfn}\label{dfn_embedded}
We say that $M$ is {\bf connected at $x\in \delta M$}\index{connected at $x$} if the set $ \{z\in M:|z-x|<\ep\}$ is connected for all  $\ep>0$ small enough.

We will say that it is {\bf connected along $Z\subset \delta M$} if it is connected at each point of $Z$. We say that $M$ is {\bf normal} if it is connected along $ \delta M$.

{\bf A $\cc^\infty$ normalization of $M$ }
 is  a subanalytic $\cc^\infty$ diffeomorphism $h: \mc\to M$ satisfying $\sup_{x\in \mc} |D_x h|<\infty$ and  $\sup_{x\in M} |D_x h^{-1}|<\infty$ (where $D_xh$ stands for the derivative of $h$),  with $\mc$ normal bounded subanalytic $\cc^\infty$ submanifold of $\R^k$, for some $k$.
 \end{dfn}

 The following proposition gathers Propositions $3.3$ and $3.4$ of \cite{trace}, yielding existence and uniqueness of $\cc^\infty$ normalizations.

 \begin{pro}\label{pro_normal_existence}\begin{enumerate}
                                         \item 
        Every bounded subanalytic $\cc^\infty$ manifold admits a $\cc^\infty$ normalization.
\item  \label{item_unique}   Every $\cc^\infty$ normalization  $h: \mc\to M$  extends continuously to a mapping from $\adh{\mc}$ to $\mba$.
    Moreover, if $h_1:\mc_1 \to M$ and $h_2:\mc_2 \to M$ are two $\cc^\infty$ normalizations of $M$, then $h_2^{-1} h_1$  extends to a homeomorphism between $\adh{\mc_1}$ and $\adh{\mc_2}$.                             \end{enumerate}

    \end{pro}
    
\end{subsection}

 \begin{subsection}{The trace on $\pau M$.}\label{sect_traceloc} In \cite{lprime}, it was shown that the trace of a function $u\in W^{1,p}(M)$ is well-defined as an $L^p_{loc}$ mapping on $\pa M$.  We start by recalling the definition in the case where $M$ is connected along $\pa M$, in which case $M\cup \pa M$ is a Lipschitz manifold with boundary, which makes it possible to rely on the classical theory, as follows.

For  $p\in [1,\infty)$,
we say that a sequence of measurable functions $u_i\in W^{1,p}(M)$ converges to $u\in W^{1,p}(M)$ {\bf in the strong $W^{1,p}_{loc}$ topology} if every $x\in M\cup \pa M$ has a neighborhood $U$ in $\R^n$ such that $u_{i|U\cap M}$ and  $\pa u_{i|U\cap M}$ respectively tend to $u_{|U\cap M}$ and $\pa u_{|U\cap M}$ in the $L^p$ norm (we say ``strong'' because it is local in $M\cup \pa M$, not only in $M$).

\begin{pro}\label{pro_trace_loc} \cite{lprime} If $M$ is connected along $\pa M$ then, for any  $p\in [1,\infty)$, 
 the mapping
 \begin{equation}\label{eq_cinfty_traceloc}\cc^\infty(\mba)\to  L^p_{loc}(\pau M), \quad \varphi\mapsto \varphi_{|\pau M}\,,\end{equation} 
 extends to a mapping
 $$\tra_{\pau M} :W^{1,p} (M)\to L^p_{loc}(\pau M)$$ 
 which is continuous in the strong $W^{1,p}_{loc}$ topology. 
\end{pro}

    In the case where $M$ is not connected along $\pa M$, one then can define a trace operator which is ``multi-valued'' (since $M$ may have several connected components locally at a generic frontier point), as follows.

 \subsection*{General definition of the trace.}
There are subanalytic subsets $S_1,\dots, S_k$ of $\pau M$ satisfying  $\hn^{m-1}(\pau M \setminus\bigcup_{i=1}^k S_{i})=0$ and such that each $S_i$ has a neighborhood $U_i$ in $\R^n$ for which each connected component of $U_i\cap M$ is the interior of a Lipschitz manifold with boundary $S_i$ (such a covering is induced for instance by the images of the simplices of a triangulation of $M$ compatible with $\pa M$).  We denote by $Z_{i,1}, \dots, Z_{i,l_i}$  these connected components.

We first define the trace on $S_i$, denoted $\tra_{S_i}$, $i\in \{1,\dots,k\}$. Applying Proposition \ref{pro_trace_loc} to $Z_{i,j}$  (which is normal if $U_i$ is sufficiently small) provides a trace operator, which induces a linear mapping
  $\tra^j_{i}:W^{1,p}(Z_{i,j})\to L^p_{loc}(S_i,\hn^{m-1})$ (since $S_i\subset \pa Z_{i,j}$).
  Define then $$\tra_{S_i} : W^{1,p}(M) \to L^p_{loc} (S_i,\hn^{m-1})^l,\qquad  l=\underset{i=1,\dots,k}\max l_i\, ,$$  by setting for $u\in W^{1,p}(M)$ and $x\in S_i$:
\begin{equation}\label{eq_tral}\tra_{S_i} u(x) := ( \tra_{i}^1 u_{|Z_{i,1}}(x), \dots , \tra_{i}^{l_i} u_{|Z_{i,l_i}}(x), 0,\dots, 0)\in \R^l.\end{equation}
    Here, we add $(l-l_i)$ times the zero function because  it will be convenient that the trace has the same number of components for all $i$.

   \begin{rem}\label{rem_tra}$ $
   \begin{enumerate}[(a)]
    \item \label{item_Lploc_tra}
   As $M\cup \pa M$ is a finite union of manifolds with boundary at every point of $\pa M$, $\tra_{S_i} u(x)$ is actually $L^p$ on a neighborhood of every point of $\adh{S_i}\cap\pa M $. In particular, the image of the trace operator is included in $L^p_{loc}(\pa M)^l$.

\item    The trace mapping of course depends on the way the elements $Z_{i,1},\dots, Z_{i,l_i}$ are enumerated. However,
up to permutation of the $l_i$ first components,  $ \tra_{S_i} u(x)$ is unique, and it only depends on the germ of $u$ near $S_i$. In particular, the kernel of this mapping is independent of any choice. 

       \item \label{item_tra_ipp} It directly follows from our local definition of the trace and Stokes' formula that  a function $u\in W^{1,p}(M)$ belongs to $W^{1,p} (M,A) $, $A$ open subset of $\pa M$, if and only if (\ref{eq_vanishing}) holds for all $\beta\in \wbf^{1,p'}_{\nabla,M\cup A}(\adh{M})$. When $M$ is connected along $\pa M$, it suffices to test  (\ref{eq_vanishing}) on the vector fields
        $\beta\in \cbb^{\infty}_{M\cup A}(\adh{M})$.
   \end{enumerate}
        \end{rem}

In \cite{lprime} is showed the following generalization of Poincar\'e inequality:

\begin{pro}\label{pro_gen_poinc1}
 Let $U$ be a nonempty open subset of $\pad M$. For every $p\in [1,\infty)$, there is a constant $C$ such that we have for each $u\in W^{1,p}(M)$:
 \begin{equation}\label{ineq_poin}
 ||u||_{L^p(M)}\le C ||\pa u||_{L^p(M)}+C||\tra_{\pa M}\, u||_{L^p(U)}.
 \end{equation} 
\end{pro}

 In \cite{lprime}, the space $W^{1,p}(M,A)$, $A$ open subset of $\pa M$, was defined as the set of those functions $u\in W^{1,p}(M)$ for which $\tra_{\pa M} u$ is zero on $A$, which is not the definition that we gave in the present article. These two definitions actually come to the same (when $A$ is an open subset of $\pa M$), as shown by the following proposition.

\begin{pro}\label{pro_WMA} If $A$ is an open subset of $\pa M$ then we have for all $p\in [1,\infty)$:
$$W^{1,p}(M,A)=\{u\in W^{1,p}(M):\tra_{\pa M} u= 0\, \mbox{ on $A$}\}. $$
\end{pro}

\begin{proof} Every $u\in W^{1,p}(M)$ that satisfies (\ref{eq_vanishing})  for all $\beta \in \wbf^{1,p'}_{\nabla,M\cup A}(M)$  must have zero trace on $A$, by Remark \ref{rem_tra} (\ref{item_tra_ipp}). Conversely,  take a function $u$ that has zero trace on an open subset $A$ of $\pa M$. Since we can take a partition of unity, it suffices to check equality (\ref{eq_vanishing}) for a vector field $\beta$ which is supported in the vicinity of a given point $\xo$ of $M\cup A$.  For such $\xo$ and $\beta$,
as $M\cup A$ is at $\xo$ a finite union of manifolds with boundary $A$, if $u$ is such that $\tra_{\pa M} u= 0$ on $A$ then (\ref{eq_vanishing}) follows from Stokes' formula.
\end{proof}

For small values of $p$, when $M$ is connected at boundary points, it is possible to approximate functions vanishing on an open subanalytic subset $Z$ of $\delta M$ by functions that are supported in the complement of $\adh Z$ and that in addition are smooth up to the boundary:

 \begin{thm}\label{thm_dense_lprime}\cite{lprime} Let $Z$ be a subanalytic open subset of $\pau M$ and let $E$ be a subanalytic subset of $ \delta M$ containing $\delta M\setminus \pau M$ and $\delta Z$, with   $\dim E\le m-2$. If $M$ is connected along $\pau M\setminus (Z \cup \adh E )$  then for all $p\in [1,m- \dim E]$ not infinite,    $\cc^\infty_{\mba\setminus (Z\cup \adh{E})}(\mba)$ is dense in $W^{1,p}(M,Z)$.
\end{thm}

 The assumption ``$M$ connected along  $\pau M\setminus (Z \cup \adh E )$'' is essential to ensure the density of functions that are smooth up to the boundary. If one is only interested in approximating by functions that are supported in the complement of a given set, as we shall need, one can always normalize the underlying manifold and then
apply the above theorem, as done in the proof of the following corollary.
\begin{cor}\label{cor_dense_lprimeb}Let a subanalytic subset $E$ of $\delta M$  satisfy $\dim E\le m-2$.
 If $A$ is an open subanalytic subset of $\pa M$ then for every $p\in [1,p_0]$ not infinite, where
 $$p_0:=m-\dim M_\sigma \cup E\cup \delta A,$$
 the space $W^{1,\infty}_{\mba \setminus  A\cup E}(M)$ is dense in $W^{1,p}(M,A)$. In particular,  for all such $p$, we have:
 $$W^{1,p}(M,A)=W^{1,p}(M,A\cup E). $$
  \end{cor}
\begin{proof} Observe that if $h:\mc \to M$ is a $\cc^\infty$ normalization of $M$, it suffices to prove the theorem for $\mc$,   $\adh  h^{-1}(E)$, and $\adh  h^{-1}(A)$, where $\adh h:\adh\mc \to \adh M$ is the extension of $h$ (see Proposition \ref{pro_normal_existence}). 
It is thus no loss of generality to assume that $M$ is normal.

Set for simplicity $$E':= M_\sigma \cup E\cup \delta A.$$ By Theorem \ref{thm_dense_lprime}, the space $\cc^\infty_{\mba\setminus (A\cup \adh{E'})}(\mba)$ is dense in $W^{1,p}(M,A)$ for all $p\in[1,m-\dim E']$ not infinite. 
\end{proof}

The above corollary enables us perform integration by parts for suitable $p$. More precisely, given a  subanalytic subset $A$ of $\delta M$ let
\begin{equation}\label{eq_pm}
 p_M(A):= m-\max(\dim  \adh A\setminus int_{\pa M} (A), \dim M_\sigma)\ge 2,
\end{equation}
where $int_{\pa M} (A)$ stands for the interior of $A\cap \pa M$ in $\pa M$. We have:

\begin{cor}\label{cor_AB}
 Let $A\subset \delta M$ be  subanalytic  and let  $p\in [1, p_M(A)]$ (not infinite). Equality (\ref{eq_vanishing}) holds for every $u\in W^{1,p}(M,A)$ and $\beta\in \wbf_\nabla ^{1,p'}(M,\pa M \setminus \adh A)$.
Consequently,
 \begin{equation*}\label{eq_spec}
 W^{1,p}(M, A)= W^{1,p}(M,\adh A).\end{equation*}
\end{cor}

\begin{proof}
Apply Corollary \ref{cor_dense_lprimeb} to $E:=\adh A\setminus int_{\pa M} (A)$ and $int_{\pa M} (A)$, where  $int_{\pa M} (A)$ is as just above.\end{proof}

\begin{rem}\label{rem_punctered}
This result is not true for $p$ large. A counterexample is provided by a holomorphic function that has a pole of order one at the origin. Such a function is an element of $\wbf_\nabla ^{1,p'}(M)$, if $p$ is sufficiently large and $M$ is a punctured disk $\{(x,y)\in \R^2:0<x^2+y^2<\ep\}$, $\ep>0$ small.  Corollary \ref{cor_pbig_vanishing} gives a variation of the above corollary in the case ``$p$ large''.

\end{rem}

 \end{subsection}

\section{Density theorems for $\wbf_\nabla^{1,p}(M,B)$.}

 The idea is to derive density theorems for  $ \wbf_\nabla ^{1,p}(M,B)$, $B$ subanalytic subset of $\delta M$,  from the density theorems of the preceding section by means of a duality argument. This requires a technical lemma for which the following remark will be of service.

\begin{rem}\label{rem_repres_functionals}
  For every $p \in (1,\infty)$, the linear mapping  $\mathbf{A}_p: W^{1,p}(M)\to L^{p}(M)^{n+1}$ defined as $\mathbf A _p (u):= (u, \pa u)$,  is continuous, injective, and has closed image. Its dual mapping  $\mathbf{A}_p':  L^{p'}(M)^{n+1}\to W^{1,p}(M)'$ is therefore onto. As a matter of fact, every  continuous linear functional $T:W^{1,p}(M)\to \R$ can be written $\alpha+\pa^* \beta$ for some  $\alpha\in L^{p'}(M)$ and $\beta\in  L^{p'}(M)^n$. We of course can choose $\beta$ in $\lbf^{p'}(M)$.
\end{rem}

\begin{lem}\label{lem_densite}
Let $A,B$, and $E$ be subanalytic subsets of $\delta M$ partitioning this set. For each $p\in (1,\infty)$, the following conditions are equivalent:
 \begin{enumerate}[(i)]
  \item  (\ref{eq_vanishing}) holds for all $u \in W^{1,p}(M,A)$ and all $\beta\in \wbf_\nabla ^{1,p'}(M,B)$.
   \item $W^{1,p}_{\mba \setminus (A\cup E)} (M)$ is dense in $W^{1,p}(M,A)$.
   \item $\wbf^{1,p'}_{\nabla,\mba \setminus (B\cup E)} (M)$ is dense in $ \wbf_\nabla ^{1,p'}(M,B)$.
 \end{enumerate}
\end{lem}
\begin{proof}
 That $(iii)$ (resp. $(ii)$) implies $(i)$ directly comes down from the definition of the space $\wbf_\nabla ^{1,p'}(M,B)$ (resp.  $W^{1,p}(M,A)$), since $$\mba\setminus (B\cup E)= M \cup A \qquad \mbox{(resp. $\mba\setminus (A\cup E)= M \cup B$).}$$
 
 To show that $(i)$ implies $(ii)$, we are going to prove, assuming $(i)$, that every continuous linear functional $T:W^{1,p}(M,A)\to \R$ that vanishes on $ W^{1,p}_{\mba \setminus (A\cup E)}(M)$ is identically zero. Such a functional $T$ can be written $\alpha +\pa^*\beta$, with $\alpha \in L^{p'}(M)$ and $\beta \in \lbf^{p'}(M)$ (see Remark \ref{rem_repres_functionals}), and we have for all $u\in  W^{1,p}_{\mba \setminus (A\cup E)}(M)$ \begin{equation}\label{eq_alpu}<\alpha,u>=-<\beta,\pa u>,\end{equation}
 which entails that $\alpha=\nabla \beta$ in the sense of distribution, and consequently that $\beta\in \wbf_\nabla^{1,p'}(M)$.  Equality (\ref{eq_alpu}) then shows that  (\ref{eq_vanishing}) holds for every $u\in W^{1,p}_{\mba \setminus (A\cup E)}(M)=W^{1,p}_{M\cup B}(M)$, which exactly means that $\beta \in \wbf^{1,p'}_\nabla (M,B)$.
 
 Equality (\ref{eq_vanishing}), that must hold for all   $u\in W^{1,p}(M,A)$ (due to $(i)$), thus yields that for all such $u$ we have $$T(u)=<\nabla \beta, u>+<\beta, \pa u>=0,$$ as required.
 
 To prove the implication $(iii)\implies (i)$, proceed in the same way, replacing $W_{\mba \setminus (A\cup E)}^{1,p} (M)$ with  $\wbf_{\nabla,\mba \setminus (B\cup E)}^{1,p'} (M)$,  $W^{1,p}(M,A)$ with   $ \wbf_\nabla ^{1,p'}(M,B)$, and $\nabla$ with $\pa$.
\end{proof} 
\begin{rem}\label{rem_dense_deltaM}$ $\begin{enumerate}[(a)]
\item  Most of the results of this section yield density of ``nicely supported'' functions and vector fields, not necessarily smooth. It is however well-known that we can approximate elements of Sobolev spaces by smooth functions (using instance de Rham regularization operators \cite{derham,gold}), and that the support of the approximation may be required to fit in any small prescribed neighborhood of the support of the considered function or vector field. As a matter of fact, the last two assertions of this lemma  provide density of smooth functions and tangent vector fields. 
                                    \item It is worthy of notice that $(iii)$ is trivial in the case $A=\delta M$, in which case, thanks to (a), we see that $(ii)$  yields that $\cc_0^\infty(M)$   is dense in $W^{1,p}(M,\delta M)$ for all $p\in (1,\infty)$. Similarly, in the case $B=\delta M$, we see that $(iii)$ of  the latter lemma establishes that $\cbb_0^{\infty}(M)$ is dense in $\wbf_\nabla ^{1,p'}(M,\delta M)$.
                                                                       \end{enumerate}

\end{rem}

\begin{thm}\label{thm_dense_MBE} Let $B$ and $E$ be subanalytic subsets of $\delta M$, and set $A:=\pa M\setminus \adh {B\cup E}$.  If $\dim E\le m-2$ then
 for all $p\in (1,q]$ (not infinite), where
 $$q:= m-\dim M_\sigma\cup \delta A\cup E,$$
  the space $\wbf^{1,p'}_{\nabla,\mba\setminus (\adh B\cup E)} (M)$ is dense in $\wbf_\nabla ^{1,p'}(M,B)$. In particular, for all such $p$, we have:
 $$\wbf_\nabla ^{1,p'}(M,B)=\wbf_\nabla ^{1,p'}(M,\adh B\cup E).$$
\end{thm}
\begin{proof}   Applying Corollary \ref{cor_dense_lprimeb} to $A$ and $E':= M_\sigma \cup \adh E\cup \delta A $, we see that the space $W^{1,p}_{\mba \setminus  A\cup E'}(M)$ is dense in $W^{1,p}(M,A)$, for $p\le m-\dim E'$.
By Lemma \ref{lem_densite}, we deduce that  $\wbf^{1,p'}_{\nabla,\mba\setminus ( B'\cup E')} (M)$ is dense in $\wbf_\nabla ^{1,p'}(M,B)$, for all such $p$, where $B':=\delta M\setminus (A\cup E')$. As $B'\cup E'\supset \adh B\cup E$, the result follows.
\end{proof}

Note that in the case where $B$ is an open subset of $\pa M$ and $E=\emptyset$, we get that for all 
 $ p\in (1, p_M(B)]$ (not infinite) 
we have:
\begin{equation*}\label{eq_spec_B}\wbf_\nabla ^{1,p'}(M,B)=\wbf_\nabla ^{1,p'}(M,\adh B).\end{equation*}

So far, we only constructed approximations that vanish in the vicinity of a prescribed set, with no attempt to make them $L^\infty$ or smooth up to the boundary (see however Remark \ref{rem_dense_deltaM} (a)). Constructing approximations in $\cc^\infty(\mba)$ indeed requires the manifold to be connected along $\pa M$ (see on this issue \cite[Corollary 3.10]{trace}). We actually have:

\begin{thm}\label{thm_dense_E}  Let $E$ be a subanalytic subset of $\delta M$ satisfying $\dim E\le m-2$. For every $p\in (1, p_M(E)]$ not infinite, we have:
\begin{enumerate}[(i)]
\item If $M$ is connected along $\pa M\setminus E$ then  the space $\cbb^{\infty}_{\mba \setminus E}(\mba)$ is dense in $\wbf_\nabla^{1,p'}(M)$.
 \item The space $\wbf^{1,\infty}_{\mba\setminus E} (M)$
 is dense in  $\wbf_\nabla^{1,p'}(M)$.
\end{enumerate}
\end{thm}
\begin{proof}  Possibly replacing $E$ with $E\cup M_\sigma$, we may assume that $E\supset  M_\sigma$.
By Theorem \ref{thm_dense_MBE} (applied with $B=\emptyset$), we know that it suffices to show that $\cc^\infty_{\mba \setminus E}(\mba)$ is dense in   $\wbf^{1,p'}_{\nabla,\mba \setminus E} (M)$ in the  $\wbf_\nabla ^{1,p'}(M)$-norm for each real number $p$ less than or equal to
$p_M(E)= m-\dim M_\sigma \cup E.$

If $M$ is connected along  $\delta M\setminus E$ and if $E\supset  M_\sigma$, then $\mba \setminus E$ is a Lipschitz manifold with boundary and the elements of $\wbf^{1,p'}_{\nabla,\mba\setminus E} (M)$ are compactly supported vector fields on it. The needed fact is therefore well-known.
%

Assertion $(ii)$ is a consequence of $(i)$, together with the facts that we can normalize $M$ and that normalizations identify the respective Sobolev spaces.
\end{proof}

\section{A trace theorem and applications}

\subsection{The cotrace operator $\gab_\nub$.}
As usual, we endow the space 
$$W^{1-\frac{1}{p}}(\pa M):=\tra_{\pa M} \left(W^{1,p}(M)\right)$$
with the norm induced by $W^{1,p}(M)$ via $\tra_{\pa M}$, and we denote by $W^{\frac{1}{p}-1}(\pa M)$ its dual space.

 At every $x\in \pa M$, the set $M\cup \pa M$ is a finite union of Lipschitz manifolds with boundary $\pa M$. Recall that the trace operator on $\pa M$ was defined (in section \ref{sect_traceloc}) by taking a finite covering of a dense subset of $\pa M$, denoted $S_1,\dots,S_k$, each $S_i$ having a neighborhood $U_i$ such that $$U_i\cap M=\bigcup_{j=1}^{l_i} Z_{i,j},$$ with $Z_{i,j}\cup S_i$ manifold with boundary for each $i$ and $j$.
 
 We denote by $\nub_j(x)$ the vector which is normal to $S_i$ at $x\in S_i$ (it is defined almost everywhere on $S_i$). We fix an orientation of every $S_i$ and will assume that $\nu_j$ is pointing inward $Z_{i,j}$ if $M$ induces on $S_i$ the orientation of $S_i$ and outward otherwise.

  We then set:
  $$\nub(x):=(\nu_1(x),\dots,\nu_{l_i}(x),\orn,\dots,\orn)\in \R^{nl}, $$
 where $\orn$ appears $(l-l_i)$ times, with $l\ge \max l_i$ (alike in the definition of the trace we add some zeros because it is more convenient that the number of components is independent of $i$).  Of course, $\nub$ depends on the way the $Z_{i,j}$'s are enumerated, and we will assume that we made it coherently with the corresponding choice that we made to define $\tra_{\pa M}$.
 
 Note that the subanalytic set
 \begin{equation}\label{eq_def_E}
  E:= \delta M\setminus \cup_{i=1} ^k S_i
 \end{equation}
has dimension less than $(m-1)$. If $M$ is normal, we choose  $k=1$, $S_1=\pa M$, $E=M_\sigma$.

Every $\beta\in \wbf^{1,\infty}_{\mba\setminus E}(M)$ is Lipschitz with respect to the inner metric, which entails that $\beta_{i,j}:=\beta_{|Z_{i,j}}$ extends continuously to $S_i$ for each $j\le l_i$. We thus can define a mapping (almost everywhere) on $\pa M$ by setting for $x\in S_i$
\begin{equation}\label{eq_gub} \gub \beta(x):= (\nub_{i,1} \cdot \beta_{i,1}(x),\dots,\nub_{i,l_i}\cdot\beta_{i,l_i}(x),0,\dots ,0)\in \R^l . \end{equation}
As customary, we shall regard $\gub \beta$ as a functional on the image of the trace operator (the elements of $W^{1-\frac{1}{p}}(\pa M)$ are $L^p_{loc}$ on $\pa M$, see Remark \ref{rem_tra} (\ref{item_Lploc_tra}).

 For the moment, we are able to define $\gub\beta$ for $\beta \in \wbf_{\mba\setminus E}^{1,\infty}(M)$.
If $\mba$ is a compact manifold with boundary, it is well-known that $\gub$ extends uniquely to a continuous operator on $ \wbf_\nabla^{1,2}(M)\to  W^{-\frac{1}{2}}(\pa M)$  \cite{t}. The theorem below establishes an analogous fact
on any bounded subanalytic manifold.

\begin{thm}\label{thm_trace_formula}
For every $p\in (1, 2]$,  $\gub$ uniquely extends to a continuous linear operator $\gub:\wbf_\nabla^{1,p'}(M)\to W^{\frac{1}{p}-1}(\pa M)$,
such that for all  $\beta\in \wbf_\nabla^{1,p'}(M)$ and
 $u\in W^{1,p}(M)$:
\begin{equation}\label{eq_trace_theorem}
 <\beta, \pa u>+<\nabla \beta,u>=<\gab_\nub \beta,\tra_{\pa M} u>.
\end{equation}
\end{thm}
\begin{proof}
  In the case where $\mba$ is a manifold with boundary,
 it immediately follows from Stokes' formula (and density of smooth forms) that (\ref{eq_trace_theorem}) holds for  any $u\in  W^{1,p}(M)$ and $\beta \in \wbf^{1,p'}_\nabla(M)$.
 In the situation of the theorem, as $\mba \setminus E$ is at every of its points a finite union of manifolds with boundary $\pa M\setminus E$,  this formula therefore must hold for all  $\beta \in \wbf_{\mba\setminus E}^{1,\infty}(M)$ whenever $u\in W^{1,p}(M)$  is supported in the vicinity of a given point of $\pa M\setminus E$.  Since we can use a partition of unity subordinated to an open covering of the closure in $\R^n$ of the support of any given such $\beta$, we see that this equality actually holds for all $u \in W^{1,p}(M)$ and  $\beta \in \wbf_{\mba\setminus E}^{1,\infty}(M)$.

  Observe now that this equality implies that for such $u$ and $\beta$
  \begin{equation*}
 |<\gab_\nub \beta,\tra_{\pa M} u>| \le  ||\beta||_{\wbf_\nabla ^{1,p'}(M)} ||u||_{W^{1,p}(M)},
  \end{equation*}
  which entails that $\gub \beta\in W^{\frac{1}{p}-1}(\pa M)$, as well as
  \begin{equation*}
   ||\gub \beta ||_{W^{\frac{1}{p}-1}(\pa M)} \le  ||\beta||_{\wbf_\nabla ^{1,p'}(M)} .
  \end{equation*}
  Theorem \ref{thm_dense_E},  which establishes the density of $\wbf_{\mba\setminus E}^{1,\infty}(\mba)$ in  $ \wbf_\nabla^{1,p'}(M)$  for all real numbers $p\le p_M(E)$  therefore clearly yields that $\gub$ uniquely extends for $p\in (1,2]$ (since $p_M(E)\ge 2$). Moreover,  (\ref{eq_trace_theorem}) continues to hold.
\end{proof}
\begin{rem}\label{rem_traceformula} We have proved that the result holds true for every $p\in (1, p_M(E)]$ not infinite (see (\ref{eq_def_E}) and (\ref{eq_pm})  for $E$ and $p_M(E)$ respectively). If $M$ is connected along $\pa M$, we therefore just need to assume $p\le m-\dim M_\sigma$ (not infinite).
 Theorem \ref{thm_trace_formula} is however no longer true for $p$ large, a counterexample being given  by the punctured disk (see Remark \ref{rem_punctered}).
\end{rem}

  This theorem enables us to show the following $\wbf^{1,p}_\nabla$ counterpart of Proposition \ref{pro_WMA}.

\begin{cor}\label{cor_WMB}
            Let $B$ be an open subset of $\pa M$. For $p\in (1,2]$, we have:
$$\wbf_\nabla^{1,p'}(M,B)=\{\beta\in \wbf_\nabla^{1,p'}(M): \gub \beta= 0  \mbox{ on $B$}  \}.$$
\end{cor}
\begin{proof}
If $u\in W^{1,p}_{M\cup B}(M )$ then $\adh {\supp \, u}\cap \delta M$ is a compact subset of $B$, which entails that $\tra_{\pa M} u$ has compact support in $B$.  The result directly follows from (\ref{eq_trace_theorem}).
\end{proof}


\subsection{The Laplace equation on subanalytic domains.} We denote by $\Delta$ the Laplace operator, i.e. $\Delta u:=\nabla \pa u$, for $u$ distribution on $M$.
\begin{thm}\label{thm_neumann}
 Let $g\in L^2(M)$ and $\theta\in W^{-\frac{1}{2}}(\pa M)$ be satisfying the compatibility condition:
 \begin{equation*}
<g,1>-<\theta,\tra_{\pa M}1>=0, \end{equation*}
where $1$ denotes the constant function equal to this number. The equation
       $$\begin{cases}
 \Delta u=g \qquad \mbox{ on } M,\\
\gab_\nub \pa u=\theta \qquad \mbox{on } \pa M,
\end{cases}$$
      has a weak solution in $W^{1,2}(M)$, which is unique up to an additive constant.
\end{thm}

\begin{proof} We rely on our trace formula (\ref{eq_trace_theorem}) to apply the classical argument.
Let $\mathbf{E}$ denote the space $W^{1,2}(M)$ quotiented by the subspace of the constant functions, and endowed with the quotient topology. By (\ref{eq_poinc_wirt}),
the bilinear form $(u,v)\mapsto <\pa u, \pa v>$ makes of $\mathbf{E}$ a Hilbert space.

As the functional $v\mapsto <\theta,\tra_{\pa M}v>- <g,v>$ is continuous on $W^{1,2}(M)$ and vanishes on $\R$ (the space of the constant functions), it defines a continuous linear functional on $\mathbf{E}$. By Riesz representation theorem,
there is $u\in \mathbf{E}$ such that for all $v$ in this space we have
 \begin{equation}\label{eq_papa}<\pa u,\pa v>=<\theta,\tra_{\pa M}v>-<g,v>,\end{equation}
which implies that for all  $\varphi\in \cc^\infty_0(M)$ we have
 \begin{equation}\label{eq_deltau}<\Delta u,\varphi>=<g,\varphi>-<\theta,\tra_{\pa M}\varphi>=<g,\varphi>.\end{equation}
  Moreover,  for any $\psi\in W^{1,2}(M)$ we have
  \begin{equation}\label{eq_gabdelta}<\gab_\nub \pa u,\tra_{\pa M}\psi>-<\Delta u,\psi>\overset{(\ref{eq_trace_theorem})}=<\pa u,\pa \psi>\overset{(\ref{eq_papa})} =<\theta,\tra_{\pa M}\psi>-<g,\psi>,\end{equation}
so that we can conclude (since $\Delta u\overset{(\ref{eq_deltau})}=g$)
\begin{equation}\label{eq_theta}<\gab_\nub \pa u,\tra_{\pa M}\psi> =<\theta,\tra_{\pa M}\psi>.\end{equation}
   Uniqueness of the solution up to a constant directly follows from (\ref{eq_trace_theorem}), which yields $\pa (u- v)=0$ if $u$ and $v$ are two solutions.
\end{proof}

\begin{thm}\label{thm_neumanndir}
Let $\Gamma_D$ and $\Gamma_N$ be subanalytic open subsets of $\pa M$, with $\Gamma_D\ne \emptyset$ and $\dim \delta M \setminus (\Gamma_D\cup \Gamma_N)\le m-2$. For any $\theta\in W^{-\frac{1}{2}}(\pa M)$, $g\in L^2(M)$, and $f\in W^{1,2}(M)$ such that $\Delta f\in L^2(M)$, the equation
  $$\begin{cases}
 \Delta u=g \qquad \qquad\mbox{ on } M,\\
 \tra_{\pa M} u= \tra_{\pa M} f \;\; \mbox{ on } \Gamma_D, \\
\; \gab_\nub \pa u=\theta \qquad \;\quad \mbox{ on } \Gamma_N,
\end{cases}$$
      has a unique weak solution in $W^{1,2}(M)$.
\end{thm}

\begin{proof}
Possibly replacing  $g$ with $(g-\Delta f)$, we can assume $f= 0$.
The argument is then analogous to the one used in the proof of the previous theorem.
By our generalized Poincar\'e-Wirtinger inequality (\ref{ineq_poin}),  the bilinear form $(u,v)\mapsto <\pa u, \pa v>$ makes of $W^{1,2}(M,\Gamma_D)$ a Hilbert space.
The functional $v\mapsto<\theta,\tra_{\pa M}v>- <g,v>$ being continuous on $W^{1,2}(M,\Gamma_D)$, it follows from Riesz representation theorem that
there is $u\in W^{1,2}(M,\Gamma_D)$ such that (\ref{eq_papa}) holds for all $v\in W^{1,2}(M,\Gamma_D)$.
This implies that (\ref{eq_deltau}) holds for all  $\varphi\in \cc^\infty_0(M)$.
  Moreover, this also entails that (\ref{eq_gabdelta}) holds for any $\psi\in W^{1,2}(M,\Gamma_D)$,
so that by (\ref{eq_deltau}), we can conclude that (\ref{eq_theta}) holds as well for such $\psi$,
which shows that $\gub \pa u= \theta$ on $\Gamma_N$.

We now turn to show uniqueness of the solution, which reduces to show that the only solution of the equation $\Delta u=0$, for the conditions $\gub \pa u= 0$ on $\Gamma_N$ and $\tra_{\pa M} u = 0$ on $\Gamma_D$, is zero. By Proposition \ref{pro_WMA} and Corollary \ref{cor_WMB}, it means that $\pa u \in \wbf_\nabla^{1,2}(M,\Gamma_N)$ and $u\in W^{1,2}(M,\Gamma_D)$, which by Corollary \ref{cor_AB} entails:
$$<\pa u,\pa u>= -<\nabla \pa u, u>=0.$$
As $\Gamma_D$ is a nonempty open subset of $\pa M$, this yields $u=0$.
\end{proof}

\begin{section}{The case where $p$ is large}
In this section, we focus on the large values of $p$ and solve an optimization problem. We first recall some needed facts.
\subsection{Morrey's embedding.}
       Given $x$ and  $y$ in $M$,  let:$$d_{M}(x,y):=\inf \{ lg(\gamma):\gamma:[0,1]\to M, \mbox{ $\cc^0$ subanalytic arc joining } x\mbox{ and } y\},$$
  where $lg(\gamma)$ stands for the length of the subanalytic arc $\gamma$ (which is piecewise $\cc^1$).
 This defines a metric on $M$, generally referred as {\bf the inner metric} of $M$.

A function $u$ on $M$ is {\bf H\"older continuous with respect to the inner metric} (with exponent $\alpha>0$) if there is  a constant $C$  such that for all $x$ and $y$ in $M$
$$ |u(x)-u(y)|\le C\; d_M(x,y)^\alpha.$$

Given $u\in \cc^{0,\alpha}(M)$ we set \begin{equation*}
                                      |u|_{\check{\cc}^{0,\alpha}(M)}:= \sup \big\{\frac{|u(x)-u(y)|}{ d_M(x,y)^\alpha}:x,y\in  M, x\ne y\big\},
                                     \end{equation*}
                                     as well as 
                                     \begin{equation*}
                                      ||u||_{\check{\cc}^{0,\alpha}(M)}:=||u||_{L^\infty(M)}+|u|_{\check{\cc}^{0,\alpha}(M)}.  
                                     \end{equation*}
                                     
When $M$ is normal (as well as bounded and subanalytic), it is possible to see that every function on $M$ that is H\"older continuous with respect to the inner metric is actually H\"older continuous with respect to the euclidean metric.

                                   We  have the following generalization of Morrey's embedding.
\begin{thm}\label{cor_morrey}\cite[Corollary $6.3$]{lprime}
There  is a positive real number $\alpha$ such that for all $p$  sufficiently large the embedding $$(\cc^\infty(\mba),||\cdot||_{W^{1,p}(M)} )\hookrightarrow (\check{\cc}^{0,\alpha}(M),||\cdot||_{\check{\cc}^{0,\alpha}(M)})$$  is continuous and therefore extends to a compact embedding $W^{1,p}(M)\hookrightarrow \check{\cc}^{0,\alpha}(M)$.
\end{thm}

\begin{subsection}{The trace when $p$ is large.}\label{sect_trace_p_large}In the case where $p$ is large, it was shown in \cite{trace} that the trace operator is $L^p$ bounded. Moreover, it was established that one can define the trace on any subanalytic subset $A$ of $\delta M$, possibly of positive codimension.
More precisely:

\begin{thm}\label{thm_trace}\cite{trace}
 Assume that $M$ is normal and let $A$ be any subanalytic subset of $\delta M$. For all $p\in [1,\infty)$ sufficiently large, we have:
\begin{enumerate}[(i)]
\item
$\cc^\infty(\adh{M})$ is dense in $W^{1,p}(M)$.
\item The linear operator
\begin{equation*}\label{trace}
\cc^\infty(\adh{M})\ni \varphi \mapsto \varphi_{|A}\in L^p(A,\hn^k), \qquad k:=\dim A,
\end{equation*}
is continuous in the norm $||\cdot ||_{W^{1,p}(M)}$ and thus extends to a mapping $\tra_A:W^{1,p}(M)\to L^p(A,\hn^k)$.
\item If  $\st$ is a stratification of $A$, then $\cc^\infty_{\adh{M}\setminus\adh{A}}(\mba)$ is a dense subspace of
$\bigcap\limits_{Y\in\st}\ker \tra_Y$.
\end{enumerate}\end{thm}

By {\bf stratification} of a set, we mean a partition of it into subanalytic manifolds.

In the case where $M$ fails to be normal, it is then possible to make use of a normalization, to obtain a trace operator $\tra_A: W^{1,p}(M)\to L^p(A,\hn^k)^l$, $k=\dim A$, with $p$ large enough and $l$ standing for the maximal number of connected components of $M$ at a point of $\delta M$,  as we did for $\tra_{\pa M}$ (see \cite[section 3]{trace} for more details).

Fix now a stratification $\st$ of $\delta M$, and define a mapping $\tra \, u: \mba \to \R^l$  by setting:
\begin{equation}\label{eq_tr}\tra \, u(x):=\begin{cases}
\tra_S u(x)\qquad \qquad \mbox{ if $x\in S\in \st$},  \\
(u(x),0,\dots,0) \quad \mbox{ whenever $x\in M$}.
  \end{cases}
\end{equation}
It easily follows from Theorem \ref{cor_morrey} that this mapping is independent of the chosen stratification.
 If $M$ is normal, one can see that $\tra\, u$ is a continuous function (for $p$ large) on $\delta M$ that extends continuously $u$.

\begin{cor}\label{cor_pbig_vanishing}
 Let $A$ be any subanalytic subset of $\delta M$. For $p$ sufficiently large,
 (\ref{eq_vanishing}) holds for all $u\in W^{1,p}(M)$ satisfying $\tra\, u= 0_{\R^l}$ on $A$  and $\beta\in \wbf_\nabla^{1,p'}(M,\delta M\setminus \adh A)$.
\end{cor}
 \begin{rem}
   Theorem \ref{thm_trace} entails that for all $p$ sufficiently large
  \begin{equation*}\label{eq_incl_paire}\{u\in W^{1,p}(M):\tra \, u=0_{\R^l} \;\;\mbox{on $A$}\}\subset W^{1,p}(M,A).\end{equation*}
  It can be seen that the reversed inclusion holds. In other words, when $p$ is large, Proposition \ref{pro_WMA} holds for any subanalytic subset $A$ of $\delta M$, possibly of positive codimension in $\delta M$. The proof of this fact, which  will not be needed, is however technical.
 \end{rem}

%

We recall that we assume that $M$ is connected.

\begin{pro}\label{pro_gen_poinc1}(Generalized Poincar\'e inequality)
 Let $A$ be a nonempty subanalytic subset of $\mba$. For every $p\in [1,\infty)$ sufficiently large, there is a constant $C$ such that we have for each $u\in W^{1,p}(M)$ and $q\in [1,\infty]$:
 \begin{equation}\label{ineq_poin_pig}
 ||u||_{L^q(M)}\le C ||\pa u||_{L^p(M)}+C||\tra\, u||_{L^p(A,\hn^k)},\quad k:=\dim A.
 \end{equation} 
\end{pro}
\begin{proof}
 We derive it from the compactness of the embedding provided by Theorem \ref{cor_morrey} using a very classical argument. Assume that (\ref{ineq_poin_pig}) fails for some   $p$ sufficiently large for the latter theorem to apply, i.e. assume that there is a sequence $u_i\in W^{1,p}(M)$ satisfying $||u_i||_{L^\infty(M)}= 1$ (it suffices to show the inequality for  $q=\infty$) and $\lim ||\tra \, u_i||_{L^p(A,\hn^k)}=\lim ||\pa u_i||_{L^p(M)}=0$. By the compactness of Morrey's embedding (Theorem \ref{cor_morrey}), extracting a subsequence if necessary,  we can assume this sequence to be convergent in $L^\infty(M)$. Its limit $u_\infty\in L^\infty(M)$ then clearly satisfies $\pa u_\infty= 0$, as distribution. As $M$ is connected, $u_\infty$ must be constant (almost everywhere), and since   $  \tra\, u_\infty = 0$ on $A\ne \emptyset$,  this constant must be zero.
\end{proof}

\end{subsection}

\subsection{Application to an optimization problem.}
 In this section, we study Problem (\ref{eq_argmin_intro}) in our framework. More precisely, given a function $f\in  W^{1,p}(M)$ (with $p$ large)  and a subanalytic subset $A$ of $\mba$, we discuss the problem  of determining (see (\ref{eq_tr}) for $\tra\, u$)
 \begin{equation}\label{eq_argmin}\argmin_{u\in W^{1,p}(M),\;\; \tra\, u= \tra\, f \mbox{ on $A$}} ||\pa u||_{L^p(M)}.\end{equation}

\medskip

It is well-known \cite{clarkson} that the space $\lbf^p(M)$ is uniformly convex for all $p\in(1,\infty)$. As a matter of fact \cite{fort}, if $E$ is a closed linear subspace of $\lbf^p(M)$ and $\ubf\in \lbf^p(M)$, there is a unique $P_E(\ubf)\in E$ such that $||\ubf-P_E(\ubf)||_{L^p(M)}$ realizes the distance to $E$, i.e.
\begin{equation*}||\ubf-P_E(\ubf)||_{L^p(M)}=\inf_{\mathbf{v}\in E} ||\ubf-\mathbf{v}||_{L^p(M)}.\end{equation*}

The only difference with the situation that occurs in Hilbert's spaces is that  $P_E$ may fail to be a linear map and $P_E^{-1}(0)$, sometimes denoted $E^\perp$, is not necessarily a vector subspace of $\lbf^p(M)$.

\begin{thm}\label{thm_opt} For every subanalytic subset $A$ of $ \mba$ and $f\in W^{1,p}(M)$,
   Problem (\ref{eq_argmin}) has a unique solution 
   $u\in W^{1,p}(M)$,
for all $p\in (1,\infty)$ sufficiently large. This solution is H\"older continuous on $M$ with respect to the inner metric.
\end{thm}
\begin{proof} Fix a nonempty subanalytic subset $A$ of $\mba$  and a function  $f\in W^{1,p}(M)$.
 Set for simplicity  
 $$Q:=\{u\in W^{1,p}(M): \tra\, u = 0\mbox{ on $A$}\} ,$$
 where $p\in [1,\infty)$ is sufficiently large for Theorem \ref{cor_morrey} to apply. 
Let 
 $E$ be the image of the mapping $\pa :Q\to \lbf^p(M), u \mapsto \pa u$. By Proposition \ref{pro_gen_poinc1}, the space $E$ is closed. As a matter of fact, there is a unique element $P_E(\pa f)\in E$ such that
 \begin{equation*}\label{eq_fPE}
||\pa f-P_E(\pa f)||_{L^p(M)} \le ||\pa f- \pa v||_{L^p(M)},
 \end{equation*}
 for all $v\in Q$. Let us thus define a function $u$ on $M$ by setting
 $$u(x):=f(x)-w(x),$$
  where $w$ is the element of $ Q$ that satisfies $\pa w=P_E(\pa f)$. 
 By Morrey's embedding (Theorem \ref{cor_morrey}), $u$ must be H\"older continuous on $M$  with respect to the inner metric. The uniqueness of the solution directly comes down from the uniqueness of the distance minimizer in uniformly convex Banach spaces.
\end{proof}
\begin{rem}\label{rem_pb}$ $
  This theorem does not assume that $f$  extends continuously to $\mba$.  Due to our generalized Morrey's embedding, any $f\in W^{1,p}(M)$, with $p$ large, nevertheless has to be continuous at the points at which  $M$ is connected. The same holds for the solution $u$. In particular, if $M$ is normal then $u$ extends  to a function on $\mba$ which is H\"older continuous with respect to the euclidean metric. 
\end{rem}

 Since $A$ may have nonempty interior in $M$, it is natural (and useful for applications) to wonder whether the solution provided by this theorem coincides on $M':=M\setminus \adh A$ with the minimizer corresponding to this manifold $M'$ subject to the constraint $\tra \, u =\tra\, f$ on $A':=A\cap \delta M'$. The answer is positive, and follows from Morrey's embedding together with the following proposition which is of its own interest. 

\begin{pro}\label{pro_ext}
 Let $Z$ be a nowhere dense subanalytic subset of $M$  and let $u \in W^{1,p}(M\setminus Z)$, $p\in [1,\infty]$. If $u$ extends to a continuous function on $M$ then $u\in W^{1,p}(M)$. 
\end{pro}

\begin{proof} We first check that $u$ is ``differentiable'' at every  of $Z^{m-1}_{reg}$ (the set of points at which $Z$ is a smooth submanifold of dimension $(m-1)$), i.e. that for every $\xo$ in this set and every smooth vector field $\varphi\in \cbb_0^{\infty}(M)$ supported in the vicinity of this point, we have
\begin{equation}\label{eq_u_diff}
 <\pa u,\varphi>=-<u,\nabla \varphi>,
\end{equation}
where $\pa u$ stands for the vector field defined on $M$ (almost everywhere, since $Z$ has measure zero) as the gradient of $u\in W^{1,p}(M\setminus Z)$.

Let $\xo\in Z^{m-1}_{reg}$. Since $M$ is at such a point $\xo$ the union of two manifolds with boundary $Z$ and our problem is local, we can assume that $M$ is a ball in $\R^m$ centered at $\xo$ and $Z$ is a hyperplane passing through $\xo$. This hyperplane divides $M$ into two connected components, say $U$ and $V$.   As $u_{|U}$ and $u_{|V}$ respectively belong to $W^{1,p}(U)$ and $W^{1,p}(V)$, and since $U$ and $V$ induce opposite orientations on $Z_{reg}^{m-1}$, to which $u$ extends continuously, we have by Stokes' formula for $\varphi$ as above:
$$ \int_U \pa u\cdot \varphi=-\int_U u  \nabla \varphi+\int_{Z^{m-1}_{reg}} u \varphi,$$
and
$$  \int_V \pa u\cdot \varphi=-\int_V u  \nabla \varphi -\int_{Z^{m-1}_{reg}}  u \varphi. $$
  Adding these two equalities clearly establishes  (\ref{eq_u_diff}). 
  
Let $B:=Z\setminus Z_{reg}^{m-1}$.
Equality (\ref{eq_u_diff}) shows that  $ u \in W^{1,p}(M\setminus B)$, which, since  $\dim B \le m-2,$  yields that $u\in W^{1,p}(M)$ (see for instance \cite[section 1.1.18]{mazya}).
\end{proof}

\medskip

 We are going to establish (Lemma \ref{lem_deltap} and Corollary \ref{cor_c1reg}) that it satisfies a Neumann type condition at the points of $\delta M\setminus \adh A$, as well as $\Delta_p u=0$ on  $M\setminus \adh A$, where $\Delta_p$ is the $p$-Laplacian (see below), which will yield that $u$ is a $\cc^1$ function on $M\setminus \adh A$ \cite{E,le,to,ur}.

  Although the argument that we shall use is indeed just an adaptation to our framework of some classical arguments, we provide details, which will give us the opportunity to shed light on the role played by the assumption ``$p$ large''.  Let us first recall  the definition of the $p$-Laplacian.
 Given $\beta\in \lbf^p(M)$, we let $$\sharp_p \beta:=\frac{|\beta|^{p-2}}{||\beta||_{L^p(M)}^{p-2}}\cdot \beta,$$
as well as for $u\in W^{1,p}(M)$
$$\Delta_p u= \nabla \sharp_p \pa u.$$
In the case $p=2$, we of course get the usual Laplace operator.

The lemma below  generalizes to subanalytic manifolds a well-known property of the $p$-Laplacian  (see for instance \cite{degen}).
\begin{lem}\label{lem_deltap}Let $A$ be a nonempty subanalytic subset of $\mba$ and let $B:=\delta M\setminus \adh A$.
 For all $p$ sufficiently large and $u\in W^{1,p}(M)$, the following conditions are equivalent:
 \begin{enumerate}
  \item $\Delta_p u=0$ on $M\setminus \adh A$, and $\,\sharp_p \pa u\in \wbf_\nabla^{1,p'}(M, B)$.
  \item For all $v\in W^{1,p}(M)$ satisfying $ \tra\, v= \tra\, u$ on $A$, we have $$||\pa u||_{L^p(M)} \le ||\pa v||_{L^p(M)} .$$
 \end{enumerate}
\end{lem}
\begin{proof}
A function $u$ satisfying $(2)$ is a minimizer of $||\pa v||_{L^p(M)}^p$, for $v$ such that  $ \tra\, (v-u)= 0$ on $A$, and therefore a critical point of the restriction of this functional to the space of such $v$. Hence, for all $w$ in $$W^{1,p}_{\mba\setminus \adh A}(M)=W^{1,p}_{(M\setminus \adh A)\cup B}(M),$$ we have
 $$<|\pa u|^{p-2} \pa u , \pa w >=0,$$
  which means that  $\Delta_p u=0$ on $M\setminus \adh A$. Moreover, vanishing on $B$ is a local condition in the sense that it only needs to be tested on functions that are supported in the vicinity of each point of $B$, that therefore may be assumed to be compactly supported in $ (M\setminus \adh A)\cup B$. This equality thus also yields that $\sharp_p \pa u \in \wbf^{1,p'}_\nabla(M,B)$.
  
  Conversely, if $u\in W^{1,p}(M)$ satisfies $(1)$, then,  by Corollary \ref{cor_pbig_vanishing}, for all $v\in W^{1,p}(M)$ satisfying $\tra\, v= \tra\, u$ on $A$,  we have
 \begin{equation}\label{eq_u-v}  <\pa (u-v) ,\sharp_p \pa u>= 0,\end{equation}
 which, by H\"older inequality, shows that
  \begin{equation}\label{eq_cauchyschwartz}
   ||\pa u||_{L^p(M)}^2=<\pa u ,\sharp_p \pa u>\overset{(\ref{eq_u-v})}= <\pa v ,\sharp_p \pa u>\le ||\pa v||_{L^p(M)}||\pa u||_{L^p(M)} ,\end{equation}
  yielding $(2)$.
\end{proof}

As emphasized in the paragraph that follows the proof of Proposition \ref{pro_ext}, this, together with Theorem \ref{thm_opt},  has the following consequence \cite{E,le,to,ur}:
\begin{cor}\label{cor_c1reg}
 The unique solution of problem (\ref{eq_argmin}) provided by Theorem \ref{thm_opt} is  $\cc^1$ on $M\setminus \adh A$. Its derivative is locally H\"older continuous on this set.
\end{cor}

\end{section}

\end{document}